\documentclass[reqno]{amsart}

\usepackage[utf8]{inputenc}
\usepackage{amsthm}
\usepackage{amsmath}
\usepackage{amssymb}
\usepackage{enumitem}
\usepackage[dvips]{graphicx}
\usepackage{color}
\usepackage{hyperref}
\usepackage{url}
\usepackage[left=3.5cm, right=3.5cm, paperheight=11.4in]{geometry}
\usepackage{fancyhdr}
\usepackage{bm}
\usepackage{mathrsfs}
\usepackage{stmaryrd}
\usepackage{nicefrac}

\newtheorem{theorem}{Theorem}
\newtheorem*{theorem*}{Main Theorem}

\theoremstyle{definition}

\theoremstyle{remark}

\hypersetup{
    pdftitle={Commutativity of quasi-arithmetic means on measure spaces},
    pdfauthor={G\l{}azowska, Leonetti, Matkowski, and Tringali},
    pdfmenubar=false,
    pdffitwindow=true,
    pdfstartview=FitH,
    colorlinks=true,
    linkcolor=blue,
    citecolor=green,
    urlcolor=black
}

\DeclareMathSymbol{\widehatsym}{\mathord}{largesymbols}{"62}

\renewcommand{\rho}{\varrho}

\newcommand{\fixed}[2][1]{%
  \begingroup
  \spaceskip=#1\fontdimen2\font minus \fontdimen4\font
  \xspaceskip=0pt\relax
  #2%
  \endgroup
}

\begin{document}
\title{Commutativity of integral quasi-arithmetic \\ means on measure spaces}

\author{Dorota G\l azowska}
\address{Faculty of Mathematics, Computer Science and Econometrics, University of Zielona G\'ora -- prof. Z. Szafrana 4a, PL-65516 Zielona G\'ora}
\email{d.glazowska@wmie.uz.zgora.pl}

\author{Paolo Leonetti}
\address{Department of Statistics, Universit\`a ``Luigi Bocconi'' -- via Roentgen 1, IT-20136 Milano}
\email{leonetti.paolo@gmail.com}
\urladdr{\url{https://sites.google.com/site/leonettipaolo/}}

\author{Janusz Matkowski}
\address{Faculty of Mathematics, Computer Science and Econometrics, University of Zielona G\'ora -- prof. Z. Szafrana 4a, PL-65516 Zielona G\'ora}
\email{j.matkowski@wmie.uz.zgora.pl}

\author{Salvatore Tringali}
\address{Institute for Mathematics and Scientific Computing, University of Graz -- Heinrichstr. 36, AT-8010 Graz}
\email{salvatore.tringali@uni-graz.at}
\urladdr{\url{http://imsc.uni-graz.at/tringali}}

%
%


\begin{abstract}
Let $(X, \mathscr{L}, \lambda)$ and $(Y, \mathscr{M}, \mu)$ be finite measure spaces for which there exist $A \in \mathscr{L}$ and $B \in \mathscr{M}$ with $0 < \lambda(A) < \lambda(X)$ and $0 < \mu(B) < \mu(Y)$, and let $I\subseteq \mathbf{R}$ be a non-empty interval. We prove that, if $f$ and $g$ are cont\-in\-u\-ous bijections $I \to \mathbf{R}^+$, then the equation
$$
f^{-1}\!\left(\int_X f\!\left(g^{-1}\!\left(\int_Y g \circ h\;d\mu\right)\right)d \lambda\right)\! = g^{-1}\!\left(\int_Y g\!\left(f^{-1}\!\left(\int_X f \circ h\;d\lambda\right)\right)d \mu\right)
$$
is satisfied by every $\mathscr{L} \otimes \mathscr{M}$-measurable simple function $h: X \times Y \to I$ if and only if $f=c \fixed[0.15]{\text{ }} g$ for some $c \in \mathbf{R}^+$ (it is easy to see that the equation is well posed).
An analogous, but es\-sen\-tial\-ly different, result, with $f$ and $g$ replaced by con\-tin\-u\-ous injections $I \to \mathbf R$ and $\lambda(X)=\mu(Y)=1$, was re\-cent\-ly obtained in [Indag. Math. \textbf{27} (2016), 945--953].
\end{abstract}
\maketitle
\thispagestyle{empty}
\section{Introduction}
\label{sec:intro}
Let $(X,\mathscr{L},\lambda)$ and $(Y,\mathscr{M},\mu)$ be measure spaces, and $f$ and $g$ be real-valued continuous in\-jec\-tions defined on a non-empty interval $I\subseteq \mathbf{R}$ (which may be bounded or unbounded, and need not be open or closed). In this note, we examine conditions under which the equation
\begin{equation}\label{eq:mainequation}
f^{-1}\!\left(\int_X f\!\left(g^{-1}\!\left(\int_Y g \circ h\;d\mu\right)\right)d \lambda\right)\! = g^{-1}\!\left(\int_Y g\!\left(f^{-1}\!\left(\int_X f \circ h\;d\lambda\right)\right)d \mu\right)\!
\end{equation}
is satisfied by every $h$ in a suitable class of $\mathscr{L} \otimes \mathscr{M}$-measurable functions $X\times Y\rightarrow I$, taking $f$ and $g$ as unknowns and assuming the equation is well posed (notations and terminology, if not explained, are standard or should be clear from the context).

When $(X,\mathscr{L},\lambda)$ and $(Y,\mathscr{M},\mu)$ are probability spaces, the left- and right-hand side of \eqref{eq:mainequation} can be interpreted as ``partially mixed'' integral quasi-arithmetic means.
The interest in functional equations involving generalized means dates back at least to G.~Aumann \cite{Auma} and has been a subject of extensive research, see, e.g.,
\cite{Dar},
\cite{Kalig}, \cite{Matk12, Matk05}, and references therein.

In particular, \eqref{eq:mainequation} is naturally related to the vast literature on permutable mappings \cite{Ritt}, and is motivated by the study of certainty equivalences, a notion first introduced by S.~H.~Chew \cite{Chew} in connection to the theory of expected utility and decision making under uncertainty, see \cite{MMR} and \cite{Str} for current trends in the area.

The equation was recently addressed in \cite{LMT16}, where it was observed, among other things, that \eqref{eq:mainequation} is well posed if $(X,\mathscr{L},\lambda)$ and $(Y,\mathscr{M},\mu)$ are probability spaces and $h(X\times Y) \Subset I$ for every ``test function'' $h$, see \cite[Proposition 2]{LMT16} (``$\Subset$'' means, as usual, ``contained in a compact subset of'').
It follows that, if $(X,\mathscr{L},\lambda)$ and $(Y,\mathscr{M},\mu)$ are probability spaces, then both the left- and the right-hand side of \eqref{eq:mainequation} is well defined provided that $h: X\times Y\rightarrow I$ is an $\mathscr{L}\otimes \mathscr{M}$-measurable simple function, namely, $h = \sum_{i=1}^n \alpha_i \bm{1}_{E_i}$, where $\alpha_1,\ldots,\alpha_n \in I$ and $E_1,\ldots,E_n \in \mathscr{L}\otimes \mathscr{M}$ are disjoint sets such that $E_1\cup \cdots \cup E_n=X\times Y$.

With this in mind, we call a measure space $(S,\mathscr{C},\gamma)$ \emph{non-degenerate} if there exists $A \in \mathscr{C}$ with $0 < \gamma(A) < \gamma(S)$.
Here, then, comes the main theorem of \cite{LMT16}, which was stated in that paper under the assumption that \eqref{eq:mainequation} is satisfied for all $\mathscr{L}\otimes \mathscr{M}$-measurable functions $h:X\times Y \to I$ for which $h(X \times Y) \Subset I$, but is actually true, as is transparent from its proof, in the following (more general) form.
\begin{theorem}\label{th:mainold}
Let $(X,\mathscr{L},\lambda)$ and $(Y,\mathscr{M},\mu)$ be non-degenerate probability spaces, and $f,g: I\to \mathbf{R}$ be continuous injections. Then equation \eqref{eq:mainequation} is satisfied by every $\mathscr{L}\otimes \mathscr{M}$-measurable simple func\-tion $h:X\times Y \to I$ if and only if $f=ag+b$ for some $a,b \in \mathbf{R}$ with $a\neq 0$.
\end{theorem}
Now we may ask what happens if $(X,\mathscr{L},\lambda)$ and $(Y,\mathscr{M},\mu)$ are not probability spaces, and in the next section we give a partial answer to this question.
\section{Main result}
It is easy to check (we omit details) that \eqref{eq:mainequation} is still well posed if $(X,\mathscr{L},\lambda)$ and $(Y,\mathscr{M},\mu)$ are non-degenerate finite measure spaces, and $f$ and $g$ are con\-tin\-u\-ous bijections $I\to \mathbf{R}^+$ (throughout, $\mathbf R^+$ denotes the set of positive reals and $\mathbf N^+$ the set of positive integers).
Accordingly, we have the following analogue of Theorem \ref{th:mainold}.
\begin{theorem}\label{th:mainnew}
Let $(X,\mathscr{L},\lambda)$ and $(Y,\mathscr{M},\mu)$ be non-degenerate finite measure spaces, and $f,g: I\to \mathbf{R}^+$ be continuous bijections, where $I \subseteq \mathbf R$ is a \textup{(}necessarily open\textup{)} interval. Then equation \eqref{eq:mainequation} is satisfied by every $\mathscr{L}\otimes \mathscr{M}$-measurable simple function $h:X\times Y \to I$ if and only if $f=c \fixed[0.15]{\text{ }}g$ for some $c \in \mathbf{R}^+$.
\end{theorem}
\begin{proof}
The ``if'' part follows by Fubini's theorem (viz., \cite[Theorem 3.4.4]{Bogac}) and the fact that, if $(S,\mathscr{C},\gamma)$ is a measure space, $w$ a continuous bijection $I\to \mathbf{R}^+$, and $h:S\to I$ a $\mathscr{C}$-measurable function such that $w\circ h$ is $\gamma$-integrable, then
$$
w^{-1}\fixed[-0.6]{\text{ }}\left(\int_S w\circ h\;d\gamma\right)=(\alpha w)^{-1}\fixed[-0.6]{\text{ }}\left(\int_S (\alpha w)\circ h\;d\gamma\right)
$$
for every $\alpha \in \mathbf{R}^+$ (we omit details,
cf. \cite[Proposition 3]{LMT16} for the case of probability spaces).

As for the ``only if'' part, set $D := \mathbf{R}^+\fixed[-0.3]{\text{ }}\times \mathbf{R}^+$. By hypothesis, there are determined $A \in \mathscr{L}$ and $B \in \mathscr{M}$ such that $\alpha_1:=\lambda(A)$, $\alpha_2:=\lambda(A^c)$, $\beta_1:=\mu(B)$, and $\beta_2:=\mu(B^c)$ belong to $\mathbf{R}^+$, where $A^c := X \setminus A$ and $B^c := Y \setminus B$. Hence, for all $x,y,z,w\in I$ the function
\begin{equation}
\label{equ:expression-of-h-as-sum-of-simple-fncts}
h=x\bm{1}_{A\times B}+y\bm{1}_{A\times B^c}+z\bm{1}_{A^c\times B} + w\bm{1}_{A^c\times B^c}
\end{equation}
is an $\mathscr{L}\otimes \mathscr{M}$-measurable simple function $X\times Y\to I$, so we can plug \eqref{equ:expression-of-h-as-sum-of-simple-fncts} into \eqref{eq:mainequation} and obtain
\begin{equation}
\label{equ:main1}
\begin{split}
f^{-1} & (\alpha_1 f(g^{-1}(\beta_1\fixed[0.2]{\text{ }} g(x)+\beta_2\fixed[0.2]{\text{ }} g(y))) + \alpha_2 f(g^{-1}(\beta_1 \fixed[0.2]{\text{ }}g(z)+\beta_2\fixed[0.2]{\text{ }}g(w)))) \\
& = g^{-1}(\beta_1\fixed[0.2]{\text{ }} g(f^{-1}(\alpha_1 f(x)+\alpha_2f(z))) + \beta_2\fixed[0.2]{\text{ }} g(f^{-1}(\alpha_1 f(y)+\alpha_2f(w)))).
\end{split}
\end{equation}
Set $\varphi:=f\circ g^{-1}$ on $\mathbf{R}^+$. Of course, $\varphi$ is a continuous bijection on $\mathbf R^+$, and we derive from \eqref{equ:main1}, through the change of variables $x \mapsto g^{-1}(s)$, $y \to g^{-1}(t)$, $z \mapsto g^{-1}(u)$, and $w \mapsto g^{-1}(v)$, that
\begin{equation}
\label{eq:main2}
\begin{split}
\varphi^{-1} (\alpha_1\fixed[0.2]{\text{ }}\varphi(\beta_1\fixed[0.2]{\text{ }} s&+\beta_2\fixed[0.2]{\text{ }}t) + \alpha_2\fixed[0.2]{\text{ }}\varphi(\beta_1\fixed[0.2]{\text{ }} u+\beta_2\fixed[0.2]{\text{ }}v)) \\
             & = \beta_1 \fixed[0.2]{\text{ }} \varphi^{-1}(\alpha_1 \fixed[0.2]{\text{ }}\varphi(s) + \alpha_2\fixed[0.2]{\text{ }}\varphi(u)) + \beta_2\fixed[0.2]{\text{ }} \varphi^{-1}(\alpha_1 \fixed[0.2]{\text{ }}\varphi(t) + \alpha_2\fixed[0.2]{\text{ }} \varphi(v))
\end{split}
\end{equation}
for every $s,t,u,v \in g(I)=\mathbf{R}^+$. Moreover, if we take $\Phi$ to be the function
\begin{equation}
\label{equ:def-of-Phi}
D \to \mathbf{R}^+: (x,y) \mapsto \varphi^{-1}(\alpha_1 \fixed[0.2]{\text{ }}\varphi(x) + \alpha_2\fixed[0.2]{\text{ }}\varphi(y)),
\end{equation}
then \eqref{eq:main2} can be conveniently rewritten as
\begin{equation}\label{eq:main3}
\Phi(\beta_1 \mathbf{x}+\beta_2 \mathbf{y})= \beta_1 \Phi(\mathbf{x})+\beta_2 \Phi(\mathbf{y}),\text{ for all }\mathbf{x}, \mathbf{y} \in D.
\end{equation}
Let $\preceq$ be the product order on $\mathbf{R} \times \mathbf{R}$ induced by the usual order on $\mathbf R$, and note that \begin{equation}
\label{equ:monotonicity-of-Phi}
\Phi(\mathbf{x})<\Phi(\mathbf{y}), \text{ for all distinct }\mathbf{x}, \mathbf{y} \in D\text{ with }\mathbf{x}\preceq \mathbf{y}.
\end{equation}
Indeed, $\varphi$ being a continuous bijection on $\mathbf R^+$ entails that $\varphi$ is strictly monotone. So, assume $\varphi$ is strictly increasing (respectively, strictly decreasing), and let $x,y, z,w \in \mathbf{R}^+$ be such that $x \le z$, $y \le w$, and $(x, y) \ne (z, w)$. Then
$$
\alpha_1\varphi(x)+\alpha_2\varphi(y)<\alpha_1\varphi(z)+\alpha_2\varphi(w) \quad(\text{respectively, } \alpha_1\varphi(x)+\alpha_2\varphi(y)>\alpha_1\varphi(z)+\alpha_2\varphi(w)),
$$
and since $\varphi$ is strictly increasing (respectively, decreasing) if and only if so is $\varphi^{-1}$, we conclude that
$\Phi(x,y)<\Phi(z,w)$, which is what we wanted to prove.

On the other hand, it is straightforward to check that $\Phi$ is surjective. Indeed, pick $z \in \mathbf R^+$. By the surjectivity of $\varphi$, there exist $x, y \in \mathbf R^+$ such that $\alpha_1\varphi(x) = \alpha_2 \varphi(y) = \frac{1}{2} \varphi(z) > 0$, viz., $\alpha_1\varphi(x) + \alpha_2 \varphi(y) = \varphi(z)$, which, by \eqref{equ:def-of-Phi}, is equivalent to $\Phi(x,y) = z$.

With this said, set $\xi_n := \Phi(1/n, 1/n)$ for every $n \in \mathbf N^+$. By \eqref{equ:monotonicity-of-Phi}, $(\xi_n)_{n \ge 1}$ is a strictly de\-creas\-ing sequence of positive reals. Hence, the limit of $\xi_n$ as $n \to \infty$ exists, and is non-negative and equal to $\xi := \inf_{n \ge 1} \xi_n$. Suppose for a contradiction that $\xi > 0$. Then, we infer from the surjectivity of $\Phi$ that $\xi = \Phi(\bar{x},\bar{y})$ for some $\bar{x}, \bar{y} \in \mathbf R^+$, which is, however, impossible, because $\frac{1}{n} < \min(\bar{x},\bar{y})$, and hence, by \eqref{equ:monotonicity-of-Phi}, $\xi_n < \xi$, for all sufficiently large $n \in \mathbf N^+$.

So, using that a local base at $\boldsymbol{0} := (0,0)$ (in the usual topology of $\mathbf{R}^2$) is given by the squares of the form $[-1/n, 1/n] \times [-1/n, 1/n]$ with $n \in \mathbf N^+$, it follows from the above that
\begin{equation*}\label{eq:limitphi}
\lim_{\mathbf{x} \to \mathbf 0}\Phi(\mathbf{x}) = 0.
\end{equation*}
By letting $\mathbf{x} \to \boldsymbol{0}$ (respectively, $\mathbf{y} \to \boldsymbol{0})$ in \eqref{eq:main3}, we therefore find that
\begin{equation*}\label{eq:main4}
\Phi(\beta_1\mathbf{x})=\beta_1\Phi(\mathbf{x})\ \text{ and }\ \Phi(\beta_2\mathbf{y})=\beta_2\Phi(\mathbf{y}), \text{ for all }\mathbf{x}, \mathbf{y} \in D.
\end{equation*}
Together with \eqref{eq:main3}, this in turn implies that
\begin{equation}
\label{equ:additivity-of-Phi}
\Phi(\mathbf{x}+\mathbf{y})=\Phi(\mathbf{x})+\Phi(\mathbf{y}),\text{ for all }\mathbf{x}, \mathbf{y} \in D.
\end{equation}
But $D$ is a subsemigroup of the group $(\mathbf R^2, +)$ with $\mathbf R^2 = D - D := \{x-y: x, y \in D\}$ and $\Phi$ is continuous, so we get from \eqref{equ:additivity-of-Phi} and \cite[Theorems 5.5.2 and 18.2.1]{Kucz} that there exist $a,b \in \mathbf{R}$ such that
$\Phi(x,y)=ax+by$ for all $x,y \in \mathbf{R}^+$, and actually, it is immediate that $a, b \ge 0$ and $a+b \ne 0$, since $\Phi$ is a positive function. In addition, we derive from \eqref{equ:def-of-Phi} that
\begin{equation}\label{eq:main5}
\alpha_1 \varphi(x)+\alpha_2\varphi(y)=\varphi(ax+by), \text{ for all }x,y \in \mathbf{R}^+.
\end{equation}
Now, we have already observed that $\varphi$ is strictly monotone. Suppose for a contradiction that $\varphi$ is strictly decreasing. Then, $\varphi$ being a bijection of $\mathbf{R}^+$ gives that $\varphi(z) \to 0^+$ as $z \to \infty$, and assuming $a \neq 0$ (the case when $b \ne 0$ is similar), this implies by \eqref{eq:main5} that
$$
0 < \alpha_2 \varphi(1) =\varphi(ax+b)-\alpha_1 \varphi(x) < \varphi(ax+b) \le \varphi(ax),
$$
which is, however, impossible in the limit as $x$ goes to $\infty$.

Thus, $\varphi$ is a strictly increasing continuous bijection of $\mathbf{R}^+$, and hence $\varphi(z) \to 0^+$ as $z \to 0$. Taking $\varphi(0) := 0$ and letting $x\to 0$ (respectively, $y\to 0$) in \eqref{eq:main5}, we can therefore conclude that
$$
\alpha_2\varphi(y)=\varphi(by) \ \text{ and }\ \alpha_1\varphi(x)=\varphi(ax),\text{ for all }x,y \in \mathbf{R}^+.
$$
It follows $a, b \in \mathbf R^+$, and in combination with \eqref{eq:main5}, this yields
$$
\varphi(x+y)=\varphi(x)+\varphi(y),\text{ for all }x,y \in \mathbf{R}^+.
$$
So, considering that $\varphi$ is con\-tin\-u\-ous and applying \cite[Theorems 5.5.2 and 18.2.1]{Kucz} to the function $\mathbf R^+ \fixed[-0.3]{\text{ }} \times \mathbf R \to \mathbf R: (x,y) \mapsto \varphi(x)$ shows that there is a constant $c \in \mathbf{R}^+$ such that $\varphi(x)=c\fixed[0.25]{\text{ }} x$ for all $x \in \mathbf R^+$, which is equivalent to $f=c\fixed[0.15]{\text{ }} g$.
\end{proof}


\section*{Acknowledgments}
P.L. was supported by a PhD scholarship from Universit\`a ``Luigi Bocconi'', and S.T. by the Austrian Science Fund (FWF), Project No. M 1900-N39.


\end{document}